\documentclass{amsart}

\usepackage{amsmath,amsthm,amssymb}
\usepackage{hyperref}
\usepackage[all]{xy}
\usepackage{tikz}

\newtheorem{thm}{Theorem}
\newtheorem{prop}[thm]{Proposition}
\newtheorem{lem}[thm]{Lemma}
\newtheorem{cor}[thm]{Corollary}

\theoremstyle{definition}
\newtheorem{dfn}[thm]{Definition}
\newtheorem{rem}[thm]{Remark}
\newtheorem{ex}[thm]{Example}

\theoremstyle{remark}
\newtheorem*{org}{Organization}
\newtheorem*{ack}{Acknowledgments}

\numberwithin{thm}{section}
\numberwithin{equation}{section}

\DeclareMathOperator{\dist}{dist}
\DeclareMathOperator{\diam}{diam}
\DeclareMathOperator{\vol}{vol}
\DeclareMathOperator{\gexp}{gexp}
\DeclareMathOperator{\rank}{rank}
\DeclareMathOperator{\const}{const}

\begin{document}

\title[Uniform boundedness on extremal subsets]{Uniform boundedness on extremal subsets in Alexandrov spaces}
\author[T. Fujioka]{Tadashi Fujioka}
\address{Department of Mathematics, Osaka University, Toyonaka, Osaka 560-0043, Japan}
\email{\href{mailto:tfujioka210@gmail.com}{tfujioka210@gmail.com}}
\date{\today}
\subjclass[2010]{53C20, 53C23}
\keywords{Alexandrov spaces, extremal subsets, Betti numbers, volume, collapse, essential coverings}

\begin{abstract}
In this paper, we study extremal subsets in Alexandrov spaces with dimension $n$, curvature $\ge\kappa$, and diameter $\le D$.
We show that the following three quantities are uniformly bounded above in terms of $n$, $\kappa$, and $D$: (1) the number of extremal subsets in an Alexandrov space; (2) the Betti numbers of an extremal subset; (3) the volume of an extremal subset.
The proof is an application of essential coverings introduced by Yamaguchi.
\end{abstract}

\maketitle

\section{Introduction}\label{sec:intro}

\subsection{Main results}\label{sec:main}
Alexandrov spaces are metric spaces with a lower sectional curvature bound in the sense of Toponogov's comparison theorem in Riemannian geometry.
Alexandrov spaces naturally arise as Gromov--Hausdorff limits of Riemannian manifolds or quotient spaces of Riemannian manifolds by isometric group actions.
Extremal subsets are singular sets in Alexandrov spaces defined in terms of critical points of distance functions, introduced by Perelman--Petrunin \cite{PP1}.
Typical examples are the boundary of an Alexandrov space and the projection of the fixed point set to the quotient space of a Riemannian manifold mentioned above.
See \S\ref{sec:pre} for the precise definitions, examples, and properties.
Although extremal subsets equipped with the induced intrinsic metrics do not generally have lower curvature bounds, they enjoy several important properties that hold for Alexandrov spaces (for instance, see \cite{PP1}, \cite{PP2}, \cite{Pet1}, \cite{Pet2}, \cite{K}, \cite{F}).

Let $\mathcal A(n,\kappa,D)$ denote the set of isometry classes of Alexandrov spaces with dimension $n$, curvature $\ge\kappa$, and diameter $\le D$.
It is well-known that $\bigcup_{k\le n}\mathcal A(k,\kappa,D)$ is compact with respect to the Gromov--Hausdorff distance.
From this fact, one would expect various quantities on Alexandrov spaces in $\mathcal A(n,\kappa,D)$ to be uniformly bounded.
The main results of this paper are the following uniform boundedness on extremal subsets.

\begin{thm}\label{thm:main}
For given $n$, $\kappa$, and $D$, there exists a constant $C=C(n,\kappa,D)$ such that the following hold for any $M\in\mathcal A(n,\kappa,D)$:
\begin{enumerate}
\item The number of extremal subsets in $M$ is not greater than $C$.
\item The total Betti number of any extremal subset $E$ of $M$ is not greater than $C$ (independent of the coefficient field).
\item The $m$-dimensional Hausdorff measure of any extremal subset $E$ of $M$ is not greater than $C$, where $m=\dim E$.
\end{enumerate}
\end{thm}

Here the Hausdorff dimension of an extremal subset is equal to the topological dimension, which is an integer (\cite[1.1(1)]{F}, \cite[4.4(2)]{A}).
Note that an extremal subset has two metrics, i.e., the induced intrinsic metric and the restriction of the original metric of the ambient Alexandrov space.
However, the Hausdorff measure is independent of which metric is chosen (\cite[3.17]{F}).

\subsection{Related results}\label{sec:rel}
We first mention related results in the literature.
As for Theorem \ref{thm:main}(1), it is already known that every compact Alexandrov space has only finitely many extremal subsets (\cite[3.6]{PP1}).
Our proof can be viewed as a refinement of the original proof of this fact.
Note that the same statement as Theorem \ref{thm:main}(1) is cited in \cite[4.5]{A} as an unpublished result of Petrunin.
Theorem \ref{thm:main}(2) is an analog of Gromov's Betti number theorem \cite{G1} for Riemannian manifolds.
Liu--Shen \cite{LS} generalized Gromov's theorem to Alexandrov spaces, and later Yamaguchi \cite{Y} gave an alternative proof using his essential coverings (discussed in the next section).
Our proof is a slight modification of Yamaguchi's one.
Regarding Theorem \ref{thm:main}(3), there is a recent stronger result of Li--Naber \cite{LN} on the volumes of singular sets of Alexandrov spaces.
More precisely, \cite[1.4]{LN} together with \cite[3.5]{F} implies Theorem \ref{thm:main}(3).
However, their proof is quite different from ours, and also our technique provides further information on the induced intrinsic metrics of extremal subsets (see Corollary \ref{cor:precpt} below).

In some special cases, the optimal constants of Theorem \ref{thm:main} and the rigidity in the equality cases are known:
\begin{itemize}
\item (Perelman \cite[4.3]{Per3}) The maximal number of extremal points in a compact $n$-dimensional Alexandrov space of nonnegative curvature is $2^n$.
The classification of the equality case was given by Lebedeva \cite{L}.
\item (W\"orner \cite[1.8]{W}) The maximal number of boundary strata of a compact  $n$-dimensional Alexandrov space of nonnegative curvature is $2n$, and the equality is attained only by a Euclidean cuboid (see \cite{W} for the definition of boundary strata, where this result is attributed to Perelman).
\item (Petrunin \cite[3.3.5]{Pet2}) The maximal volume of the boundary of an $n$-dimensional Alexandrov space with curvature $\ge1$ is equal to the volume of the standard unit sphere of dimension $n-1$.
The equality case was classified by Grove--Petersen \cite{GP} and the general curvature bound case was studied by Deng--Kapovitch \cite{DK}.
Note that it is unknown whether the boundary of an Alexandrov space equipped with the induced intrinsic metric is an Alexandrov space with the same lower curvature bound.
\end{itemize}

\subsection{Main tools}\label{sec:tool}
We next discuss the proof of Theorem \ref{thm:main}.
Our main tools are essential coverings and isotopy covering systems introduced by Yamaguchi \cite{Y}, which are related to the collapsing of Alexandrov spaces.
We defer the precise definitions to \S\ref{sec:ess} and here we give an example illustrating these concepts.
Consider a thin rectangle $M=[0,1]\times[0,\varepsilon]$, where $\varepsilon\ll 1$.
This is an Alexandrov space of nonnegative curvature, and each edge and each vertex are extremal subsets.
For any $p\in M$, the distance function from $p$ has critical points in its closed $\varepsilon$-neighborhood (except $p$; see \S\ref{sec:pre} for the definition of critical point).
Therefore, if one tries to cover $M$ by metric balls on which distance functions from the centers have no critical points, the minimal number of such balls grows with the order $\varepsilon^{-1}$ as $\varepsilon\to0$, which cannot be uniformly bounded.
For example, metric balls of radius $4\varepsilon/5$ centered at $\varepsilon$-discrete points on the long edges forms such a covering (see Figure \ref{fig:ess1}).

\begin{figure}[h]
\centering
\begin{tikzpicture}
\draw[thick](-3,-0.5)rectangle(3,0.5);

\coordinate(a0)at(-3,0.5);
\coordinate(a1)at(-2,0.5);
\coordinate(a2)at(-1,0.5);
\coordinate(a3)at(0,0.5);
\coordinate(a4)at(1,0.5);
\coordinate(a5)at(2,0.5);
\coordinate(a6)at(3,0.5);

\coordinate(b0)at(-3,-0.5);
\coordinate(b1)at(-2,-0.5);
\coordinate(b2)at(-1,-0.5);
\coordinate(b3)at(0,-0.5);
\coordinate(b4)at(1,-0.5);
\coordinate(b5)at(2,-0.5);
\coordinate(b6)at(3,-0.5);

\draw[<->](-3.5,0.5)to[edge node={node[midway,fill=white]{$\varepsilon$}}](-3.5,-0.5);
\draw[<->](-3,-1)to[edge node={node[midway,fill=white]{$1$}}](3,-1);
\draw[white](3.5,0.5)to(3.5,-0.5);

\fill(a0)circle(1.5pt);
\fill(a1)circle(1.5pt);
\fill(a2)circle(1.5pt);
\fill(a3)circle(1.5pt);
\fill(a4)circle(1.5pt);
\fill(a5)circle(1.5pt);
\fill(a6)circle(1.5pt);

\fill(b0)circle(1.5pt);
\fill(b1)circle(1.5pt);
\fill(b2)circle(1.5pt);
\fill(b3)circle(1.5pt);
\fill(b4)circle(1.5pt);
\fill(b5)circle(1.5pt);
\fill(b6)circle(1.5pt);

\draw([shift=(a0)]-100:0.8)arc(-100:10:0.8);
\draw([shift=(a1)]-190:0.8)arc(-190:10:0.8);
\draw([shift=(a2)]-190:0.8)arc(-190:10:0.8);
\draw([shift=(a3)]-190:0.8)arc(-190:10:0.8);
\draw([shift=(a4)]-190:0.8)arc(-190:10:0.8);
\draw([shift=(a5)]-190:0.8)arc(-190:10:0.8);
\draw([shift=(a6)]-190:0.8)arc(-190:-80:0.8);

\draw([shift=(b0)]100:0.8)arc(100:-10:0.8);
\draw([shift=(b1)]190:0.8)arc(190:-10:0.8);
\draw([shift=(b2)]190:0.8)arc(190:-10:0.8);
\draw([shift=(b3)]190:0.8)arc(190:-10:0.8);
\draw([shift=(b4)]190:0.8)arc(190:-10:0.8);
\draw([shift=(b5)]190:0.8)arc(190:-10:0.8);
\draw([shift=(b6)]190:0.8)arc(190:80:0.8);
\end{tikzpicture}
\caption{}\label{fig:ess1}
\end{figure}

\begin{figure}[h]
\centering
\begin{tikzpicture}
\draw[thick](-3,-0.5)rectangle(3,0.5);

\coordinate[label=left:$p_1$](p1)at(-3,0);
\coordinate[label=left:$p_{11}$](p11)at(-3,0.5);
\coordinate[label=left:$p_{12}$](p12)at(-3,-0.5);
\coordinate[label=right:$p_2$](p2)at(3,0);
\coordinate[label=right:$p_{21}$](p21)at(3,0.5);
\coordinate[label=right:$p_{22}$](p22)at(3,-0.5);

\fill(p1)circle(1.5pt);
\fill(p11)circle(1.5pt);
\fill(p12)circle(1.5pt);
\fill(p2)circle(1.5pt);
\fill(p21)circle(1.5pt);
\fill(p22)circle(1.5pt);

\draw(-30:1.5)arc(-30:30:1.5)node[above]{$B_1$};
\draw[dashed]([shift=(p1)]-90:0.6)arc(-90:90:0.6)node[above]{$\hat B_1$};
\draw([shift=(p11)]-100:0.9)arc(-100:10:0.9)node[above]{$B_{11}$};
\draw([shift=(p12)]100:0.9)arc(100:-10:0.9)node[below]{$B_{12}$};

\draw(-150:1.5)arc(-150:-210:1.5)node[above]{$B_2$};
\draw[dashed]([shift=(p2)]-90:0.6)arc(-90:-270:0.6)node[above]{$\hat B_2$};
\draw([shift=(p21)]-80:0.9)arc(-80:-190:0.9)node[above]{$B_{21}$};
\draw([shift=(p22)]80:0.9)arc(80:190:0.9)node[below]{$B_{22}$};
\end{tikzpicture}
\caption{}\label{fig:ess2}
\end{figure}

However, we can cover $M$ by a much smaller number of metric balls with similar properties, as follows (see Figure \ref{fig:ess2}).
Let $p_1$, $p_2$ be the midpoints of the short edges and $p_{11}$, $p_{12}$, $p_{21}$, $p_{22}$ the four vertices such that $p_{ij}$ are adjacent to $p_i$, as shown in the figure.
First, we cover $M$ by two metric balls $B_i$ of radius $2/3$ centered at $p_i$.
Then the distance function from $p_i$ has no critical points in $B_i\setminus\hat B_i$, where $\hat B_i$ is a concentric ball of radius $2\varepsilon/3$ (the dashed arcs in the figure).
Second, we cover these $\hat B_i$ by four metric balls $B_{ij}$ of radius $5\varepsilon/6$ centered at $p_{ij}$.
Then the distance function from $p_{ij}$ has no critical points in $B_{ij}$ (except $p_{ij}$).
In summary:
\begin{enumerate}
\item $M$ is covered by $B_i$;
\item The critical points of the distance function from $p_i$ are covered by $B_{ij}$;
\item $B_{ij}$ is free of critical points of the distance function from $p_{ij}$.
\end{enumerate}
Such a multi-step covering is called an \textit{isotopy covering system} and the last four balls $B_{ij}$ are called an \textit{essential covering} of $M$.
In general, one can extend this construction to more than two steps (i.e., if $B_{ij}$ still contains critical points, then cover them by $B_{ijk}$).
The number of steps needed to obtain balls with no critical points is called \textit{depth}.
Yamaguchi \cite{Y} proved that for any $M\in\mathcal A(n,\kappa,D)$, the minimal number of metric balls forming an essential covering of $M$ with depth $\le n$ is uniformly bounded above by $C(n,\kappa,D)$ (Theorem \ref{thm:ess}).

Once such a covering is obtained, the proof of Theorem \ref{thm:main} reduces to how to control each quantity (the number of extremal subsets, the Betti number of an extremal subset, and the volume of an extremal subset) on this covering.
The general outline is as follows.
Let $\chi$ denote the quantity under consideration.
First we give a bound on $\chi(B_{ij})$, using the fact that the essential covering $B_{ij}$ is critical point free.
This yields a bound on $\chi(\hat B_i)$, as $\hat B_i$ is covered by $B_{ij}$.
Next we give a bound on $\chi(B_i)$ in terms of $\chi(\hat B_i)$, using the fact that the annulus $B_i\setminus\hat B_i$ is critical point free.
Finally, since $M$ is covered by $B_i$, we obtain a bound on $\chi(M)$.

To carry out the above argument, we use yet another tools.
For Theorem \ref{thm:main}(1), we use the original technique of Perelman--Petrunin \cite[3.6]{PP1}.
For Theorem \ref{thm:main}(2), we use the fibration theorem and the stability theorem of Perelman \cite{Per1}, \cite{Per2}, Perelman--Petrunin \cite{PP1}, and Kapovitch \cite{K}.
For Theorem \ref{thm:main}(3), we use the gradient-exponential map of Perelman--Petrunin \cite{PP2}.

\subsection{Remarks and Corollaries}\label{sec:rem}
To obtain the conclusion of Theorem \ref{thm:main}, it is necessary to fix the dimension of ambient Alexandrov spaces, and it is not sufficient to fix the dimension of extremal subsets.
For example, consider the following $n$-dimensional Alexandrov space of curvature $\ge1$:
\[M_n=\left\{(x_1,\dots,x_{n+1})\in\mathbb R^{n+1}\mid x_1^2+\dots+x_{n+1}^2=1,\ x_1,\dots,x_n\ge0\right\}.\]
For each $1\le i\le n$, the minimal geodesic $\gamma_i$ between $(0,\dots,0,1)$ and $(0,\dots,0,-1)$ passing through $(0,\dots,\underset i1,\dots,0,0)$ is a one-dimensional extremal subset of $M_n$.
Therefore the number of one-dimensional extremal subsets in $M_n$ is not uniformly bounded when $n\to\infty$.
Similarly, $E_n=\bigcup_{i=1}^n\gamma_i$ is a one-dimensional extremal subset of $M_n$, but neither the Betti number nor the volume of $E_n$ is not uniformly bounded when $n\to\infty$.
The reason this happens, even though the dimension of extremal subsets is fixed, is because the dimension of $M_n$ is not fixed.

There are two corollaries of the main theorem.
For nonnegatively curved spaces, Theorem \ref{thm:main} (1) and (2) hold without the upper diameter bound $D$.

\begin{cor}\label{cor:nonneg}
For given $n$, there exists a constant $C(n)$ such that the following hold for any Alexandrov space $M$ of nonnegative curvature:
\begin{enumerate}
\item The number of extremal subsets in $M$ is not greater than $C(n)$.
\item The total Betti number of any extremal subset $E$ of $M$ is not greater than $C(n)$ (independent of the coefficient field).
\end{enumerate}
\end{cor}

In the same way as Theorem \ref{thm:main}(3), we get a uniform bound (depending on $\varepsilon$) on the number of $\varepsilon$-discrete points in an extremal subset of an Alexandrov space in $\mathcal A(n,\kappa,D)$ (Theorem \ref{thm:dis}).
This holds for the induced intrinsic metric of an extremal subset.
Therefore we obtain the following precompactness theorem.
Let $\mathcal E(n,\kappa,D)$ denote the set of isometry classes of connected extremal subsets of Alexandrov spaces in $\mathcal A(n,\kappa,D)$, equipped with the induced intrinsic metrics.

\begin{cor}\label{cor:precpt}
$\mathcal E(n,\kappa,D)$ is precompact with respect to the Gromov--Hausdorff distance.
\end{cor}

In other words, any sequence of extremal subsets in $\mathcal E(n,\kappa,D)$ has a convergent subsequence.
A natural question is: what is the limit space?
It is known that the limit is also an extremal subset with the induced intrinsic metric, provided that the sequence of the ambient Alexandrov spaces does not collapse.
More precisely, if Alexandrov spaces $M_i$ converge to $M$ without collapse and extremal subsets $E_i\subset M_i$ converge to $E\subset M$ as subsets, then $E$ is extremal (\cite[4.1.3]{Pet2}) and the induced intrinsic metrics of $E_i$ converge to that of $E$ (\cite[1.2]{Pet1}).
On the other hand, it is unknown what the limit space is when the ambient spaces collapse.

Finally, we remark that Yamaguchi's paper \cite{Y} is still unpublished, but at least this is not due to any mistake.
In fact we review his arguments in detail and provide self-contained proofs.
This paper can be read without consulting \cite{Y}.

\begin{org}
In \S\ref{sec:pre}, we review the basics of Alexandrov spaces and extremal subsets.
We also recall several notions related to the gradient-exponential map used in \S\ref{sec:vol}.
In \S\ref{sec:ess}, we define essential coverings and isotopy covering systems and review the main results of \cite{Y}.
In \S\ref{sec:num}, we prove the uniform boundedness of the numbers of extremal subsets in Alexandrov spaces (Theorem \ref{thm:main}(1) and Corollary \ref{cor:nonneg}(1)).
In \S\ref{sec:betti}, we prove the uniform boundedness of the Betti numbers of extremal subsets (Theorem \ref{thm:main}(2) and Corollary \ref{cor:nonneg}(2)).
In \S\ref{sec:vol}, we prove the uniform boundedness of the volumes of extremal subsets (Theorem \ref{thm:main}(3) and Corollary \ref{cor:precpt}).
\end{org}

\begin{ack}
I would like to thank Professor Takao Yamaguchi for his advice and encouragement.
This paper (and \cite{F}) are based on my master's thesis.
I am also grateful to the referees for their valuable comments.
This work was partially supported by JSPS KAKENHI Grant Number 22KJ2099.
\end{ack}

\section{Preliminaries}\label{sec:pre}

The distance between $p$ and $q$ is denoted by $|pq|$ or $d(p,q)$.
We denote by $B(p,r)$ and $\bar B(p,r)$ the open and closed metric ball of radius $r$ centered at $p$, respectively.
For $0<r_1<r_2$, $A(p;r_1,r_2)$ denotes the closed metric annulus $\bar B(p,r_2)\setminus B(p,r_1)$.
For a metric space $(X,d)$ and $\lambda>0$, $\lambda X$ denotes the rescaled space $(X,\lambda d)$.

\subsection{Alexandrov spaces}\label{sec:alex}
Here we review the basics of Alexandrov spaces.
We refer to \cite{BGP}, \cite{BBI}, and \cite{AKP} for details.

A \textit{geodesic space} is a metric space such that every two points can be joined by a minimal geodesic.
We assume that every minimal geodesic is parametrized by arclength.
For $\kappa\in\mathbb R$, the \textit{$\kappa$-plane} is the complete, simply-connected surface of constant curvature $\kappa$.
For three points $p$, $q$, and $r$ in a geodesic space, we consider a geodesic triangle on the $\kappa$-plane with side lengths $|pq|$, $|pr|$, and $|qr|$.
We denote by $\tilde\angle qpr$ the angle opposite to $|qr|$ and call it the \textit{comparison angle} at $p$.

A complete geodesic space $M$ is called an \textit{Alexandrov space} with curvature $\ge\kappa$ if every point has a neighborhood $U$ satisfying the following: any two minimal geodesics $\gamma$ and $\sigma$ in $U$ starting at the same point $p$,  the comparison angle $\tilde\angle\gamma(t)p\sigma(s)$ is nonincreasing in both $t$ and $s$.
In this paper, we only deal with finite-dimensional Alexandrov spaces in the sense of Hausdorff dimension.
The Hausdorff dimension of an Alexandrov space is equal to the topological dimension.
From now, $M$ denotes an $n$-dimensional Alexandrov space.

For any two minimal geodesics $\gamma$ and $\sigma$ starting at $p\in M$, one can define their \textit{angle} $\angle(\gamma,\sigma):=\lim_{t,s\to0}\tilde\angle\gamma(t)p\sigma(s)$.
The angle $\angle$ is a pseudo-distance on the space $\Gamma_p$ consisting of all minimal geodesics starting at $p$. 
The completion of the metric space induced from $(\Gamma_p,\angle)$ is called the \textit{space of directions} at $p$ and denoted by $\Sigma_p$.
$\Sigma_p$ is a compact $(n-1)$-dimensional Alexandrov space with curvature $\ge1$.
The Euclidean cone $K(\Sigma_p)$ over $\Sigma_p$ is called the \textit{tangent cone} at $p$ and denoted by $T_p$.
$(T_p,o)$ is isometric to the pointed Gromov--Hausdorff limit $\lim_{\lambda\to\infty}(\lambda M,p)$, where $o$ denotes the vertex of the cone.
$T_p$ is an $n$-dimensional Alexandrov space of nonnegative curvature.

For $p,q\in M$, we denote by $q'_p\in\Sigma_p$ one of the directions of minimal geodesics from $p$ to $q$.
Similarly, for a closed subset $A\subset M$, we denote by $A'_p\subset\Sigma_p$ the set of all directions of minimal geodesics from $p$ to $A$.

For notational simplicity, here we assume that the lower curvature bound is $-1$.
Let $\mathcal A(n)$ denote the set of isometry classes of $n$-dimensional Alexandrov spaces with curvature $\ge-1$, and $\mathcal A(n,D)$ its restriction to all elements with diameter $\le D$. 
Similarly, let $\mathcal A_p(n)$ denote the set of isometry classes of $n$-dimensional pointed Alexandrov spaces with curvature $\ge-1$.
The following property is the starting point of this paper.

\begin{thm}[{\cite[\S8]{BGP}}]\label{thm:precpt}
$\mathcal A(n,D)$ (resp.\ $\mathcal A_p(n)$) is precompact with respect to the Gromov--Hausdorff topology (resp.\ the pointed Gromov--Hausdorff topology), i.e., any sequence has a convergent subsequence.
The limit is an Alexandrov space with dimension $\le n$ and curvature $\ge-1$.
\end{thm}

\subsection{Extremal subsets}\label{sec:ex}
Here we review the basics of extremal subsets.
We refer to \cite{PP1}, \cite[\S4]{Pet2}, and \cite{F} for details.

Let $M$ be an Alexandrov space.
We denote by $\dist_q$ the distance function $d(q,\,\cdot\,)$ from $q\in M$.
A point $p\in M\setminus\{q\}$ is called a \textit{critical point} of $\dist_q$ if
\[\min_{pq}\angle(q'_p,\xi)\le\pi/2\]
for any $\xi\in\Sigma_p$, where $pq$ runs over all minimal geodesics from $p$ to $q$.
By the first variation formula, this is equivalent to $d_p\dist_q(\xi)\le 0$ (see \S\ref{sec:semi} for the differential).

\begin{dfn}\label{dfn:ex}
A closed subset $E$ of an Alexandrov space $M$ is said to be \textit{extremal} if the following condition is satisfied:
\begin{itemize}
\item[($*$)] If $\dist_q|_E$ has a local minimum at $p\in E$, where $q\in M\setminus E$, then $p$ is a critical point of $\dist_q$.
\end{itemize}
Note that $\emptyset$ and $M$ are regarded as extremal subsets of $M$.

Moreover, if $M$ has curvature $\ge1$, the following additional condition is imposed: if $E$ is empty or a singleton $\{p\}$, then we require $\diam M\le\pi/2$ or $M\subset\bar B(p,\pi/2)$, respectively.
This definition is used only for the space of directions and is stated explicitly when used.
\end{dfn}

\begin{ex}\label{ex:ex}
Let $M$ be an Alexandrov space.
\begin{enumerate}
\item A singleton $\{p\}\subset M$ is extremal (in the sense of the standard definition ($*$)) if and only if $\diam\Sigma_p\le\pi/2$.
It is called an \textit{extremal point}.
\item (\cite[1.2]{PP1}) Any Alexandrov space admits a \textit{canonical stratification}:
there exists a sequence of closed subsets uniquely determined by the topological structure of $M$,
\[M=M_n\supset M_{n-1}\supset\dots\supset M_0\supset M_{-1}=\emptyset,\]
where $n=\dim M$, such that the $k$-dimensional stratum $M^{(k)}=M_k\setminus M_{k-1}$ is a $k$-dimensional topological manifold if nonempty.
Roughly speaking, $M^{(k)}$ consists of points where the conical neighborhood topologically splits off $\mathbb R^k$ but not $\mathbb R^{k+1}$ (see \cite{Per2} for the precise definition).
Then the closure of each stratum is an extremal subset.
In particular, the boundary of an Alexandrov space (the closure of $M^{(n-1)}$) is an extremal subset.
\item (\cite[4.2]{PP1}) If a compact group $G$ acts on $M$ isometrically, the quotient space $M/G$ is also an Alexandrov space with the same lower curvature bound.
For a closed subgroup $H$ of $G$, the projection of the fixed point set of $H$ to $M/G$ is an extremal subset.
\end{enumerate}
\end{ex}

For an extremal subset $E\subset M$ and $p\in E$, the \textit{space of directions} $\Sigma_pE$ of $E$ at $p$ is defined as the subset of $\Sigma_p$ consisting of all limit directions $\lim_{i\to\infty}(p_i)'_p$, where $p_i\in E\setminus\{p\}$ converges to $p$.
$\Sigma_pE$ is an extremal subset of $\Sigma_p$ (regarded as a space of curvature $\ge1$; see Definition \ref{dfn:ex}).
The subcone $K(\Sigma_pE)$ of $T_p=K(\Sigma_p)$ is called the \textit{tangent cone of $E$} at $p$ and denoted by $T_pE$ (if $\Sigma_pE=\emptyset$, we set $K(\emptyset)=\{o\}$).
$T_pE$ is isometric to the limit $\lim_{\lambda\to\infty}(\lambda E,p)$ under the convergence $(\lambda M,p)\xrightarrow{\mathrm{GH}}(T_p,o)$.
$T_pE$ is an extremal subset of $T_p$.

The union, intersection, and closure of the difference of two extremal subsets are also extremal subsets.
For example, if $E$ and $F$ are extremal subsets of $M$, then $\overline{E\setminus F}$ is an extremal subset of $M$.
Moreover, it holds that $\Sigma_p(\overline{E\setminus F})=\overline{\Sigma_pE\setminus \Sigma_pF}$ and $T_p(\overline{E\setminus F})=\overline{T_pE\setminus T_pF}$ for any $p\in\overline{E\setminus F}$.

The Hausdorff dimension of an extremal subset is equal to the topological dimension.
As in the case of an Alexandrov space, any extremal subset admits a canonical stratification uniquely determined by its topological structure.
The closure of each stratum is also an extremal subset (compare Example \ref{ex:ex}(2)).

\subsection{Semiconcave functions, gradient curves, and radial curves}\label{sec:semi}
Here we recall several notions related to the gradient-exponential map, which will be used in \S\ref{sec:vol}.
Since these will not be used in the other sections, the reader is recommended to skip this section until they read \S\ref{sec:vol}.
We refer to \cite[\S1--3]{Pet2}, \cite[\S3]{PP2}, and \cite[Ch.16]{AKP} for details.

\subsubsection*{Semiconcave functions}
Let $M$ be an Alexandrov space and $\Omega$ an open subset of $M$.
Suppose $M$ has no boundary.
For $\mu\in\mathbb R$, a (locally Lipschitz) function $f:\Omega\to\mathbb R$ is said to be \textit{$\mu$-concave} if for any minimal geodesic $\gamma(t)$ parametrized by arclength, $f\circ\gamma(t)-(\mu/2)t^2$ is concave.
When $M$ has nonempty boundary, $f$ is said to be \textit{$\mu$-concave} if its tautological extension to the double of $M$ is $\mu$-concave in the above sense.
A function $f$ is said to be \textit{semiconcave} if for any $p\in\Omega$, there exists $\mu_p\in\mathbb R$ such that $f$ is $\mu_p$-concave in some neighborhood of $p$.
For example, the distance function $\dist_q$ from $q\in M$ is semiconcave on $M\setminus\{q\}$.

For a semiconcave function $f:\Omega\to\mathbb R$ and $p\in\Omega$, its \textit{differential} $d_pf:T_p\to\mathbb R$ at $p$ is defined by $d_pf:=\lim_{\lambda\to\infty}\lambda(f-f(p))$, where $\lambda(f-f(p))$ is defined on $\lambda M$ and the limit is taken under the convergence 
$(\lambda M,p)\xrightarrow{\mathrm{GH}}(T_p,o)$.
We then define the \textit{gradient} $\nabla_pf\in T_p$ of $f$ at $p$ by
\begin{equation*}
\nabla_pf:=
\begin{cases}
d_pf(\xi_{\max})\xi_{\max}&\text{if $\max d_pf|_{\Sigma_p}>0$,}\\
\hfil o&\text{if $\max d_pf|_{\Sigma_p}\le0$.}
\end{cases}
\end{equation*}
where $\xi_{\max}\in\Sigma_p$ is a unique maximum point of $d_pf|_{\Sigma_p}$.
A point $p\in\Omega$ is called a \textit{critical point} of $f$ if $\nabla_pf=o$.
For the distance function $\dist_q$, this coincides with the definition of critical point in \S\ref{sec:ex}.

\subsubsection*{Gradient curves}
For a semiconcave function $f:\Omega\to\mathbb R$, a curve $\alpha(t)$ in $\Omega$ satisfying
\[\alpha^+(t)=\nabla_{\alpha(t)}f\]
is called a \textit{gradient curve} of $f$.
Here $\alpha^+(t)$ denotes the right tangent vector of $\alpha$ defined as the limit $\lim_{\lambda\to\infty}\alpha(t+\lambda^{-1})$ under the convergence $(\lambda M,p)\xrightarrow{\mathrm{GH}}(T_p,o)$ (the above definition assumes that it exists).
If $f$ is a $\mu$-concave function defined on $M$, then for any $p\in M$ there exists a unique gradient curve $\alpha_p:[0,\infty)\to M$ of $f$ with $\alpha_p(0)=p$.
We define the \textit{gradient flow} $\Phi_f^t:M\to M$ of $f$ by
\[\Phi_f^t(p):=\alpha_p(t)\]
for $p\in M$ and $t\ge0$.
(For a general semiconcave function, its gradient flow is not necessarily defined for all $p\in M$ and $t\ge 0$.)

\subsubsection*{Radial curves}
In what follows, we assume that the lower curvature bound is $-1$.
For $p\in M$ and $\xi\in\Sigma_p$, we consider a curve $\beta_\xi:[0,\infty)\to M$ satisfying the following differential equation:
\begin{equation}\label{eq:rad}
\begin{gathered}
\beta_\xi^+(s)=\frac{\tanh|p\beta_\xi(s)|}{\tanh s}\nabla_{\beta_\xi(s)}\dist_p,\\
\beta_\xi(0)=p,\quad\beta_\xi^+(0)=\xi.
\end{gathered}
\end{equation}
We call it the \textit{radial curve} starting at $p$ in the direction $\xi$.
For any initial data $(p,\xi)$, there exists a unique radial curve.
If there is a minimal geodesic starting in the direction $\xi$, then it coincides with the radial curve $\beta_\xi$.

In fact, radial curves starting at $p$ are reparametrizations of gradient curves of a semiconcave function $f=\cosh\circ\dist_p-1$.
Let $\alpha(t)$ be the gradient curve of $f$ starting at $x\in M\setminus\{p\}$ and $\beta(s)$ the radial curve starting at $p$ in the direction $x'_p$.
Then $\alpha(t)=\beta(s)$ for $t\ge0$ and $s\ge|px|$, where the relation between the two parameters is given by
\begin{equation}\label{eq:re}
\frac{dt}{ds}=\frac1{\tanh s\cosh|p\beta(s)|}.
\end{equation}
Note that $\beta(s)$ is well-defined for $s\ge|px|$, independent of the choice of $x'_p$.

We next explain two comparison properties of radial curves.
To do this, we define two comparison angles for $1$-Lipschitz curves.
For a $1$-Lipschitz curve $c$ and a point $p$, we consider a geodesic triangle on the $\kappa$-plane with side lengths $|pc(t_1)|$, $|t_2-t_1|$, and $|pc(t_2)|$.
We denote by $\tilde\angle pc(t_1)\!\smile\!c(t_2)$ the angle opposite to $|pc(t_2)|$.
Similarly, for $1$-Lipschitz curves $c_1$ and $c_2$ with $c_1(0)=c_2(0)=p$, we consider a geodesic triangle on the $\kappa$-plane with side lengths $|t_1|$, $|t_2|$, and $|c_1(t_1)c_2(t_2)|$.
We denote by $\tilde\angle c_1(t_1)\!\smile\!p\!\smile\!c_2(t_2)$ the angle opposite to $|c_1(t_1)c_2(t_2)|$.
In case such a triangle does not exist, we define the comparison angle to be $0$.

\begin{prop}[{\cite[3.3, 3.3.3]{PP2}}]\label{prop:rad}
Let $M$ be an Alexandrov space with curvature $\ge-1$ and $p\in M$.
\begin{enumerate}
\item For a radial curve $\beta_\xi$ starting at $p$ in the direction $\xi\in\Sigma_p$ and $q\in M$, the comparison angle $\tilde\angle qp\!\smile\!\beta_\xi(s)$ is nonincreasing in $s$.
In particular,
\[\tilde\angle qp\!\smile\!\beta_\xi(s)\le\min_{pq}\angle(q'_p,\xi),\]
where the minimum is taken over all minimal geodesics from $p$ to $q$.
\item For two radial curves $\beta_1$ and $\beta_2$ starting at $p$ such that $\beta_1|_{[0,a_1]}$ and $\beta_2|_{[0,a_2]}$ are minimal geodesics, we have
\[\tilde\angle\beta_1(s_1)\!\smile\!p\!\smile\!\beta_2(s_2)\le\tilde\angle\beta_1(a_1)p\beta_2(a_2)\]
whenever $s_1\ge a_1$ and $s_2\ge a_2$.
\end{enumerate}
\end{prop}

Now we define the \textit{gradient-exponential map} $\gexp_p:T_p\to M$ at $p\in M$ by
\[\gexp_p(s\xi):=\beta_\xi(s),\]
where $\xi\in\Sigma_p$ and $s\ge 0$.
Clearly $\gexp_p$ is an extension of the usual exponential map $\mathrm{exp}_p$ defined by minimal geodesics.
By Proposition \ref{prop:rad}(2) above, $\gexp_p$ is a $1$-Lipschitz map from $(T_p,\mathfrak h)$ to $M$, where $\mathfrak h$ denotes the metric on $T_p$ defined by the hyperbolic law of cosines instead of the Euclidean one (i.e., the elliptic cone over $\Sigma_p$ in the sense of \cite[4.3.2]{BGP}).

Finally, we discuss the relationship between extremal subsets and gradient/radial curves.
Let $E$ be an extremal subset of $M$.
For any semiconcave function $f$ on $M$, the gradient curve of $f$ starting at a point of $E$ remains in $E$, that is, $\Phi_f^t(E)\subset E$ (conversely, a subset with this property is extremal).
In particular, for radial curves, we have the following: if $p\in E$, then $\gexp_p(T_pE)\subset E$.

\begin{rem}\label{rem:rad}
When the lower curvature bound is $-1$, there is another definition of radial curve.
Namely, we can replace the differential equation \eqref{eq:rad} by a simpler (and slower) one
\[\beta_\xi^+(s)=\frac{\sinh|p\beta_\xi(s)|}{\sinh s}\nabla_{\beta_\xi(s)}\dist_p.\]
Then Proposition \ref{prop:rad} also holds for this curve.
See \cite[p.246]{AKP} or the printed version of \cite[\S3.2]{Pet2} (not the arXiv version).
However, we need the faster one \eqref{eq:rad} for our application.
\end{rem}

\section{Essential coverings and isotopy covering systems}\label{sec:ess}

In this section, we review the main results of Yamaguchi \cite{Y}. 

From now on, the lower curvature bound is assumed to be $-1$ (the general case is obtained by rescaling).
Recall that $\mathcal A_p(n)$ denotes the set of pointed $n$-dimensional Alexandrov spaces with curvature $\ge -1$.
For a point $p$ and $0<r_1<r_2$, $A(p;r_1,r_2)$ denotes the closed metric annulus $\bar B(p,r_2)\setminus B(p,r_1)$.
For a metric space $(X,d)$ and $\lambda>0$, $\lambda X$ denotes the rescaled space $(X,\lambda d)$.

The following rescaling theorem plays a key role throughout the paper.

\begin{thm}[{\cite[3.2]{Y}}]\label{thm:res}
Suppose $(M_i,p_i)\in\mathcal A_p(n)$ converges to an Alexandrov space $(X,p)$ with dimension $\ge1$.
Then for sufficiently small $r>0$, there exists $\hat p_i\in M_i$ converging to $p$ such that either (1) or (2) holds:
\begin{enumerate}
\item There is a subsequence $\{j\}\subset\{i\}$ such that $\dist_{\hat p_j}$ has no critical points on $\bar B(\hat p_j,r)\setminus\{\hat p_j\}$.
\item There exists a sequence $\delta_i\to0$ such that
\begin{itemize}
\item[(i)] for any $\lambda>1$ and sufficiently large $i$, $\dist_{\hat p_i}$ has no critical points on $A(\hat p_i;\lambda\delta_i,r)$;
\item[(ii)] for any limit $(Y,y_0)$ of  a subsequence of $(\frac1{\delta_i}M_i,\hat p_i)$, we have
\[\dim Y\ge\dim X+1.\]
\end{itemize}
\end{enumerate}
In particular, if $\dim X=n$, then (1) holds for all sufficiently large $i$.
\end{thm}

\begin{rem}\label{rem:res}
As can be seen from the proof, the choice of $r$ depends only on the limit space $X$ (and the dimension $n$) and is independent of the sequence $M_i$.
Indeed, it suffices to choose $r$ small enough that the rescaled space $r^{-1}X$ is sufficiently close to the tangent cone $T_p$.
The value $\delta_i$ is the maximum distance between $\hat p_i$ and critical points of $\dist_{\hat p_i}$ in $\bar B(\hat p_i,r)\setminus\{\hat p_i\}$.
\end{rem}

\begin{ex}\label{ex:res}
Let $S^1_\varepsilon$ denote the circle of length $\varepsilon\ll1$.
\begin{enumerate}
\item Consider a collapsing sequence $(K(S^1_\varepsilon),o)\xrightarrow{\mathrm{GH}}(\mathbb R_+,0)$ as $\varepsilon\to0$, where $K(\cdot)$ denotes the Euclidean cone.
If $p=0\in\mathbb R_+$, then we can choose $\hat p_i=o\in K(S^1_\varepsilon)$ so that Theorem \ref{thm:res}(1) holds.
\item Consider a collapsing sequence $\mathbb R\times S^1_\varepsilon\xrightarrow{\mathrm{GH}}\mathbb R$ as $\varepsilon\to0$.
If $p=0\in\mathbb R$, then we can choose $\delta_\varepsilon=\varepsilon/2$ so that Theorem \ref{thm:res}(2) holds ($\hat p_i$ is an arbitrary sequence converging to $p$).
\end{enumerate}
\end{ex}

Here we omit the proof of Theorem \ref{thm:res} because later we will prove a stronger quantitative version, Theorem \ref{thm:res2} (or see the original proof in \cite[3.2]{Y}).

Now we give the definitions of essential covering and isotopy covering system.
Although our definitions are slightly stronger than the original ones in \cite{Y}, we use the same terminology as in \cite{Y}.
See Remark \ref{rem:ess} for more details.

Let $M$ be an Alexandrov space.
For an open metric ball $B\subset M$ centered at $p$, we call a concentric open metric ball $\hat B\subset B$ an \textit{isotopic subball} of $B$ if $\dist_p$ has no critical points on the annulus $\bar B\setminus\hat B$.
Consider a family of open metric balls $\mathcal B=\{B_{\alpha_1\cdots\alpha_k}\}$, where
\[1\le\alpha_1\le N_1,\quad1\le\alpha_2\le N_2(\alpha_1),\quad\dots,\quad1\le\alpha_k\le N_k(\alpha_1\cdots\alpha_{k-1})\]
and $1\le k\le l$ for some $l$ depending on $\alpha_1,\alpha_2,\dots$.
We call $N_1$ the \textit{first degree} of $\mathcal B$ and $N_k(\alpha_1\cdots\alpha_{k-1})$ the \textit{$k$-th degree} of $\mathcal B$ with respect to $\alpha_1\cdots\alpha_{k-1}$.
Let $A$ be the set of all multi-indices $\alpha_1\cdots\alpha_k$ such that $B_{\alpha_1\cdots\alpha_k}\in\mathcal B$, and $\hat A$ the set of all maximal multi-indices in $A$.
Here $\alpha_1\cdots\alpha_l$ is maximal if there are no $\alpha_{l+1}$ with $\alpha_1\cdots\alpha_l\alpha_{l+1}\in A$.
For each $\alpha=\alpha_1\cdots\alpha_k\in A$, we set $|\alpha|:=k$.

\begin{dfn}\label{dfn:ess}
Let $X$ be a subset of $M$.
We call $\mathcal B$ an \textit{isotopy covering system} of $X$ if it satisfies the following conditions:
\begin{enumerate}
\item $\{B_{\alpha_1}\}_{\alpha_1=1}^{N_1}$ covers $X$;
\item  $\{B_{\alpha_1\cdots\alpha_k}\}_{\alpha_k=1}^{N_k(\alpha_1\cdots\alpha_{k-1})}$ covers an isotopic subball $\hat B_{\alpha_1\cdots\alpha_{k-1}}$ of $B_{\alpha_1\cdots\alpha_{k-1}}$;
\item for each $\alpha\in\hat A$, $\dist_{p_\alpha}$ has no critical points on $\bar B_\alpha\setminus\{p_\alpha\}$, where $p_\alpha$ is the center of $B_\alpha$;
\item there is a uniform bound $d$ such that $|\alpha|\le d$ for all $\alpha\in A$.
\end{enumerate}
We call $\mathcal U=\{B_\alpha\}_{\alpha\in\hat A}$ an \textit{essential covering} of $X$.
In addition, we call $d_0=\max_{\alpha\in\hat A}|\alpha|$ the \textit{depth} of both $\mathcal B$ and $\mathcal U$.
\end{dfn}

\begin{rem}\label{rem:ess}
The above definition is stronger than Yamaguchi's original definition.
Roughly speaking, Yamaguchi's definition only requires that $B_\alpha$ is homeomorphic to $\hat B_\alpha$ for $\alpha\in A\setminus\hat A$ and is homeomorphic to some Euclidean cone for $\alpha\in\hat A$ (see \cite[\S4]{Y} for more details).
It follows from Perelman's stability theorem and fibration theorem that our definition above implies these properties (see Theorem \ref{thm:stab}).
In simple terms, Yamaguchi's essential covering is \textit{topological}, whereas ours is \textit{geometrical}.
For the rectangle example in \S\ref{sec:tool}, the topological essential covering can be chosen to be $M$ itself, but the geometrical essential covering must be the four balls centered at the vertices, as the vertices are extremal points.
\end{rem}

For a positive integer $d$, we denote by $\tau_d(X)$ the minimal number of metric balls forming an essential covering of $X$ with depth $\le d$.
For an open metric ball $B$ in $M$ having a proper isotopic subball, we set
\[\tau_d^*(B):=\min_{\hat B}\tau_d(\hat B),\]
where $\hat B$ runs over all isotopic subballs of $B$.
In addition, if $\dist_p$ has no critical points on $\bar B\setminus\{p\}$, where $p$ is the center of $B$, we set $\tau_0^*(B):=1$; otherwise $\tau_0^*(B):=\infty$.
Then the following holds: if $X$ is covered by open metric balls $\{B_{\alpha_1}\}_{\alpha_1=1}^{N_1}$ having proper isotopic subballs, we have
\[\tau_d(X)\le\sum_{\alpha_1=1}^{N_1}\tau_{d-1}^*(B_{\alpha_1})\]
for any $d\ge1$.

\begin{ex}\label{ex:ess}
For $0<\varepsilon\ll1$, consider a thin $n$-dimensional cuboid
\[I^n_\varepsilon=[0,1]\times[0,\varepsilon]\times[0,\varepsilon^2]\times\dots\times[0,\varepsilon^{n-1}].\]
Note that the faces of each dimension are extremal subsets.
As in \S\ref{sec:tool}, we see that metric balls of radii slightly less than $\varepsilon^{n-1}$ centered at the vertices form an essential covering of $I^n_\varepsilon$ with depth $n$.
Therefore $\tau_n(I^n_\varepsilon)\le2^n$ for any $\varepsilon$ (actually the equality holds since the vertices are extremal points).
On the other hand, $\lim_{\varepsilon\to0}\tau_{n-1}(I^n_\varepsilon)=\infty$.
\end{ex}

The following theorem is the main result of \cite{Y}.
Recall that $\mathcal A(n)$ denotes the set of (non-pointed) $n$-dimensional Alexandrov spaces with curvature $\ge-1$.

\begin{thm}[{\cite[4.4]{Y}}]\label{thm:ess}
For given $n$ and $D$, there exists a constant $C(n,D)$ satisfying the following:
for any $M\in\mathcal A(n)$ and $p\in M$, we have
\[\tau_n(B(p,D))\le C(n,D).\]
\end{thm}

\begin{rem}\label{rem:ess2}
More precisely, \cite[4.4]{Y} states that there exists an isotopy covering system of $B(p,D)$ whose first degree is bounded above by $C(n,D)$ and whose other higher degrees are bounded above by some constant $C(n)$ independent of $D$.
However, the above simple version is enough for our applications.
\end{rem}

Let us review the proof, assuming Theorem \ref{thm:res}.

\begin{proof}
For $1\le k\le n$, we prove the following two statements by reverse induction on $k$.

\begin{itemize}
\item[($P_k$)]\textit{Suppose $(M_i,p_i)\in\mathcal A_p(n)$ converges to a $k$-dimensional Alexandrov space $(X,p)$.
Then we have
\[\liminf_{i\to\infty}\tau_{n-k+1}(B(p_i,D))<\infty.\]}
\item[($Q_k$)]\textit{Suppose $(M_i,p_i)\in\mathcal A_p(n)$ converges to a $k$-dimensional Alexandrov space $(X,p)$.
Then for sufficiently small $r>0$, there exists a sequence $\hat p_i\in M_i$ converging to $p$ such that
\[\liminf_{i\to\infty}\tau_{n-k}^*(B(\hat p_i,r))<\infty.\]}
\end{itemize}

Here $r$ and $\hat p_i$ in ($Q_k$) are the ones of Theorem \ref{thm:res}.
In particular, $r$ depends only on the limit space $X$ (see Remark \ref{rem:res}).
Note that ($P_k$) is a global claim while ($Q_k$) is a local claim.
The proof is carried out in the following alternating order: $(Q_n)\Rightarrow(P_n)\Rightarrow\dots\Rightarrow(Q_1)\Rightarrow(P_1)$.

($Q_n$) clearly follows from Theorem $\ref{thm:res}(1)$.
Let us prove $(Q_k)\Rightarrow(P_k)$.
Suppose that ($P_k$) does not hold.
Then there exists $(M_i,p_i)\in\mathcal A_p(n)$ converging to $(X,p)$ with $\dim X=k$ such that
\[\lim_{i\to\infty}\tau_{n-k+1}(B(p_i,D))=\infty.\]
By compactness, we can cover $\bar B(p,D)$ by finitely many balls $\{B(x_\alpha,r_\alpha/2)\}_{\alpha=1}^N$, where $r_\alpha$ is the one of ($Q_k$).
Then there exist a subsequence $\{j\}$ and a constant $C$ such that $\tau_{n-k}^*(B(\hat x_\alpha^j,r_\alpha))\le C$ for every $\alpha$ and some $\hat x_\alpha^j\to x_\alpha$.
Since $\{B(\hat x_\alpha^j,r_\alpha)\}_{\alpha=1}^N$ is a covering of $B(p_j,D)$ for sufficiently large $j$, we have
\[\tau_{n-k+1}(B(p_j,D))\le NC.\]
This contradicts the assumption.

Next we prove $(P_n),\dots,(P_{k+1})\Rightarrow(Q_{k})$.
Suppose $(M_i,p_i)\in\mathcal A_p(n)$ converges to $(X,p)$ with $\dim X=k$.
By Theorem \ref{thm:res}, for sufficiently small $r>0$, there exists $\hat p_i\to p$ such that either (1) or (2) holds.
When (1) holds, the claim is trivial.
When (2) holds, there exists $\delta_i\to0$ satisfying both (i) and (ii).
Passing to a subsequence $\{j\}$, we may assume that $(\frac1{\delta_j}M_j,\hat p_j)$ converges to $(Y,y_0)$.
Then we have $l:=\dim Y\ge\dim X+1=k+1$.
Applying ($P_l$) to $\frac1{\delta_j}B(\hat p_j,2\delta_j)$ and passing to a subsequence again, we have
\[\tau_{n-l+1}\left(\frac1{\delta_j}B(\hat p_j,2\delta_j)\right)\le C\]
for some constant $C$.
Since $B(\hat p_j,2\delta_j)$ is an isotopic subball of $B(\hat p_j,r)$ for sufficiently large $j$, we obtain
\[\tau_{n-k}^*(B(\hat p_j,r))\le\tau_{n-k}(B(\hat p_j,2\delta_j))\le C.\]
This completes the inductive proof of ($P_k$) and ($Q_k$).

Now Theorem \ref{thm:ess} follows from ($P_1$), \dots, ($P_n$) by contradiction.
Note that the case $\dim X=0$ follows from the case $\dim X\ge1$ by rescaling $M_i$ so that the new diameter is $1$.
\end{proof}

\section{Numbers of extremal subsets in Alexandrov spaces}\label{sec:num}

In this section, we prove Theorem \ref{thm:main}(1) and Corollary \ref{cor:nonneg}(1).

For a subset $X$ of an Alexandrov space $M$, we define $\nu(X)$ as follows:
\begin{gather*}
\nu(X):=\#\Bigl(\left\{\text{$E$ : an extremal subset of $M$}\right\}\bigm/\sim\Bigr),\\
\text{where}\quad E\sim E'\iff X\cap E=X\cap E',
\end{gather*}
i.e., the number of extremal subsets in $M$ counted by ignoring the differences outside $X$.
If $X$ is covered by $\{X_\alpha\}_{\alpha=1}^N$, we have
\[\nu(X)\le\prod_{\alpha=1}^N\nu(X_\alpha).\]

The following lemma was essentially used in \cite{PP1} to show the finiteness of the number of extremal subsets in a compact Alexandrov space.
This controls the behavior of $\nu$ on the balls of isotopy covering systems.

\begin{lem}[cf.\ {\cite[3.6]{PP1}}]\label{lem:num}
Let $M$ be an Alexandrov space and $p\in M$.
\begin{enumerate}
\item If $\dist_p$ has no critical points on $\bar B(p,r)\setminus\{p\}$, then
\[\nu(B(p,r))\le\nu(\Sigma_p)+1.\]
Here $\nu(\Sigma_p)$ denotes the number of extremal subsets in $\Sigma_p$ regarded as a space of curvature $\ge1$ (see Definition \ref{dfn:ex}).
\item If $\dist_p$ has no critical points on $A(p;r_1,r_2)$, then
\[\nu(B(p,r_1))=\nu(B(p,r_2)).\]
\end{enumerate}
\end{lem}

\begin{proof}
First we show (2).
Suppose $\nu(B(p,r_1))<\nu(B(p,r_2))$ and choose extremal subsets $E,F\subset M$ such that
\[B(p,r_1)\cap E=B(p,r_1)\cap F\quad\text{and}\quad B(p,r_2)\cap E\neq B(p,r_2)\cap F.\]
We may assume that $B(p,r_2)\cap(E\setminus F)\neq\emptyset$.
Then $G=\overline{E\setminus F}$ is an extremal subset (see \S\ref{sec:ex}), which satisfies $B(p,r_1)\cap G=\emptyset$ and $B(p,r_2)\cap G\neq\emptyset$.
In particular, a point $q\in G$ closest to $p$ lies in $A(p;r_1,r_2)$.
Since $G$ is extremal, $q$ must be a critical point of $\dist_p$.
This contradicts the assumption of (2).

Next we show (1).
It follows from (2) and the assumption of (1) that every extremal subset intersecting $B(p,r)$ must contain $p$.
Let $E,F\subset M$ be extremal subsets intersecting $B(p,r)$ such that $\Sigma_pE=\Sigma_pF$.
Then
\[\Sigma_p(\overline{E\setminus F})=\overline{\Sigma_pE\setminus\Sigma_pF}=\emptyset\quad\text{and}\quad\Sigma_p(\overline{F\setminus E})=\emptyset\]
(see \S\ref{sec:ex}).
Therefore, $E$ and $F$ coincide in a sufficiently small neighborhood of $p$.
Again by (2), we see that $B(p,r)\cap E=B(p,r)\cap F$.
Thus we conclude that $\nu(B(p,r))\le\nu(\Sigma_p)+1$.
Note that the $+1$ on the right-hand side comes from the empty set of $M$ (or equivalently, extremal subsets not intersecting $B(p,r)$).
\end{proof}

\begin{rem}\label{rem:eq}
The equality in Lemma \ref{lem:num}(1) does not hold generally.
For example, consider a solid square and round the corners except for one vertex $p$, as shown in the following figure.
The resulting space $M$ is an Alexandrov space and its proper extremal subsets are the boundary and $p$.
However, $\Sigma_p$, which is isometric to $[0,\pi/2]$, contains three proper extremal subsets: the two boundary directions and their union.
Therefore, $\nu(M)=4$ and $\nu(\Sigma_p)+1=6$ (note that we are counting the empty set and the whole space as extremal subsets; see also the additional condition in Definition \ref{dfn:ex}).
\begin{figure}[h]
\centering
\begin{tikzpicture}
\draw[rounded corners=10pt](0,0)--(1,0)--(1,1)--(0,1)--(0,0);
\coordinate[label=left:$p$](p)at(0,0);
\fill(p)circle(1.5pt);
\draw[->](0,0)--(0.5,0);
\draw[->](0,0)--(0,0.5);
\end{tikzpicture}
\end{figure}
\end{rem}

Theorem \ref{thm:ess} and Lemma \ref{lem:num} imply the uniform boundedness of the numbers of extremal subsets.

\begin{thm}\label{thm:num}
For given $n$ and $D$, there exists a constant $C(n,D)$ satisfying the following:
for any $M\in\mathcal A(n)$ and $p\in M$, we have
\[\nu(B(p,D))\le C(n,D).\]
\end{thm}

\begin{proof}
We use induction on $n$.
By Theorem $\ref{thm:ess}$, there exists an isotopy covering system $\mathcal B=\{B_{\alpha_1\cdots\alpha_k}\}$ of $B(p,D)$ with depth $\le n$ whose degrees $N_k$ are bounded above by $C(n,D)$.
Let $\mathcal U=\{B_\alpha\}_{\alpha\in\hat A}$ be the essential covering associated with $\mathcal B$.

For $\alpha=\alpha_1\cdots\alpha_l\in\hat A$ and $1\le k\le l$, we prove by reverse induction on $k$ that
\[\nu(B_{\alpha_1\cdots\alpha_k})\le C(n,D).\]
In the case $k=l$, this follows from Lemma \ref{lem:num}(1) and the hypothesis of the induction on $n$.
Consider the case $k\le l-1$.
Recall that $\{B_{\alpha_1\cdots\alpha_{k+1}}\}_{\alpha_{k+1}=1}^{N_{k+1}(\alpha_1\cdots\alpha_k)}$ is a covering of an isotopic subball $\hat B_{\alpha_1\cdots\alpha_k}$ of $B_{\alpha_1\cdots\alpha_k}$.
The hypothesis of the reverse induction gives $\nu(B_{\alpha_1\cdots\alpha_{k+1}})\le C(n,D)$ for every $1\le \alpha_{k+1}\le N_{k+1}(\alpha_1\cdots\alpha_k)$.
Therefore, by Lemma \ref{lem:num}(2), we have
\[\nu(B_{\alpha_1\cdots\alpha_k})=\nu(\hat B_{\alpha_1\cdots\alpha_k})\le\prod_{\alpha_{k+1}=1}^{N_{k+1}(\alpha_1\cdots\alpha_k)}\nu(B_{\alpha_1\cdots\alpha_{k+1}})\le C(n,D)^{C(n,D)}.\]
This completes the induction step of the reverse induction.

Finally, since $\{B_{\alpha_1}\}_{\alpha_1=1}^{N_1}$ is a covering of $B(p,D)$, we conclude $\nu(B(p,D))\le C(n,D)$.
\end{proof}

Corollary \ref{cor:nonneg}(1) immediately follows from the following lemma.

\begin{lem}\label{lem:nonneg}
Let $M$ be a noncompact Alexandrov space of nonnegative curvature and $p\in M$.
Then for sufficiently large $R>0$, $\dist_p$ has no critical points on $M\setminus B(p,R)$.
\end{lem}

\begin{proof}
For a nonnegatively curved space, we have $\lim_{\lambda\to0}(\lambda M,p)=(K(M(\infty)),o)$, where the right-hand side denotes the Euclidean cone over the ideal boundary of $M$ (see \cite[1.1]{S}).
Since $\dist_o$ has no critical points on $K(M(\infty))\setminus\{o\}$, so does $\dist_p$ on $M\setminus B(p,\lambda^{-1})$ for sufficiently small $\lambda$.
\end{proof}

\begin{proof}[Proof of Corollary \ref{cor:nonneg}(1)]
Let $M$ be an $n$-dimensional (noncompact) Alexandrov space of nonnegative curvature and $p\in M$.
By rescaling, we may assume that the constant of Theorem \ref{thm:num} is independent of $D$.
Namely, there exists $C(n)$ such that $\nu(B(p,D))\le C(n)$ for any $D>0$.
Furthermore, Lemma \ref{lem:nonneg} and Lemma \ref{lem:num}(2) imply that the number of extremal subsets does not increase outside a sufficiently large ball $B(p,R)$.
Thus we have $\nu(M)\le C(n)$.
\end{proof}

\begin{rem}\label{rem:num}
As mentioned in \S\ref{sec:rel}, Perelman \cite[4.3]{Per3} showed that the number of extremal points in a compact $n$-dimensional Alexandrov space of nonnegative curvature is at most $2^n$.
On the other hand, Yamaguchi \cite[4.8]{Y} conjectured that $\tau_n\le2^n$ for such spaces (for his topological essential covering; see Remark \ref{rem:ess}).
Note that if an extremal point exists, then it must be the center of a metric ball of our geometrical essential covering.
Therefore, if Yamaguchi's conjecture is true for our geometrical essential covering, then this implies Perelman's result.
\end{rem}

\section{Betti numbers of extremal subsets}\label{sec:betti}

In this section, we prove Theorem \ref{thm:main}(2) and Corollary \ref{cor:nonneg}(2).

We need the fibration theorem and the stability theorem of Perelman \cite{Per1}, \cite{Per2}, especially their generalizations to extremal subsets by Perelman--Petrunin \cite{PP1} and Kapovitch \cite{K}.
The following is a special case of these two theorems (see the above references for the general statements).

\begin{thm}[{\cite[\S2]{PP1}}, {\cite[\S9]{K}}]\label{thm:stab}
Let $M$ be an Alexandrov space, $E\subset M$ an extremal subset, and $p\in M$.
\begin{enumerate}
\item If $\dist_p$ has no critical points on $\bar B(p,r)\setminus\{p\}$, then $B(p,r)\cap E$ is homeomorphic to $T_pE$.
Note that if $B(p,r)\cap E\neq\emptyset$, then $p\in E$ (see Lemma \ref{lem:num}).
\item If $\dist_p$ has no critical points on $A(p;r_1,r_2)$, then $A(p;r_1,r_2)\cap E$ is homeomorphic to $\partial B(p,r_1)\cap E\times[0,1]$.
\end{enumerate}
\end{thm}

The following lemma was used in the original work of Gromov \cite{G1} on the Betti numbers of Riemannian manifolds.

\begin{lem}[{\cite[Appendix]{G1}}, {\cite[5.4]{C}}]\label{lem:rank}
Let $B_{\alpha}^i$, $1\le\alpha\le N$, $0\le i\le n+1$, be open subsets of a topological space $X$ such that $\bar B_{\alpha}^i\subset B_{\alpha}^{i+1}$.
Set $A^i=\bigcup_{\alpha=1}^N B_{\alpha}^i$.
In what follows we only consider homology groups of dimension $\le n$.
For each $\mu=(\alpha_1,\dots,\alpha_m)$, let $f_\mu^i:H_*(B_{\alpha_1}^i\cap\dots\cap B_{\alpha_m}^i)\to H_*(B_{\alpha_1}^{i+1}\cap\dots\cap B_{\alpha_m}^{i+1})$ be the inclusion homomorphism.
Then the rank of the inclusion homomorphism $H_*(A^0)\to H_*(A^{n+1})$ is bounded above by the sum
\[\sum_{0\le i\le n,\mu}\rank f_\mu^i.\]
Note that if $B_{\alpha_1}^i\cap\dots\cap B_{\alpha_m}^i=\emptyset$, then we define $\rank f_\mu^i:=0$.
\end{lem}

Using the above theorem and lemma, we can show the uniform boundedness of the Betti numbers of extremal subsets.
The proof is exactly the same as in the case of Alexandrov spaces in \cite[\S5]{Y}.
Let $\beta(\ ;\mathcal F)$ denote the total Betti number $\sum_{i=0}^\infty b_i(\ ;\mathcal F)$ with respect to a coefficient field $\mathcal F$.

\begin{thm}\label{thm:betti}
For given $n$ and $D$, there exists a constant $C(n,D)$ satisfying the following:
for any $M\in\mathcal A(n)$ and extremal subset $E\subset M$ with diameter $\le D$, we have
\[\beta(E;\mathcal F)\le C(n,D),\]
where $\mathcal F$ is an arbitrary field.
\end{thm}

Note that $b_i(E;\mathcal F)=0$ for all $i>m=\dim E$ (this follows from the general dimension theory).
In what follows, we omit $\mathcal F$ and only consider homology groups of dimension $\le m$.

\begin{proof}
For an open metric ball $B$ of radius $r$, let $\lambda B$ denote the concentric open metric ball of radius $\lambda r$.
By Theorem \ref{thm:ess}, there exists an isotopy covering system $\mathcal B=\{B_{\alpha_1\cdots\alpha_k}\}$ of $E$ with depth $\le n$ whose degrees $N_k$ are bounded above by $C(n,D)$.
Set $\lambda_i:=10^i$ and $B_{\alpha_1\cdots\alpha_k}^i:=\lambda_iB_{\alpha_1\cdots\alpha_k}$ for $0\le i\le m+1$.
In view of Theorem \ref{thm:res}(2)(i) and the proof of Theorem \ref{thm:ess}, we may assume that
\begin{itemize}
\item $B_{\alpha_1\cdots\alpha_k}^{m+1}\subset B_{\alpha_1\cdots\alpha_{k-1}}$ for $1\le\alpha_k\le N_k(\alpha_1\cdots\alpha_{k-1})$;
\item $B_{\alpha_1\cdots\alpha_k}^i$ is an isotopic subball of $B_{\alpha_1\cdots\alpha_k}^{i+1} $ for $0\le i\le m$.
\end{itemize}
Let $\mathcal U=\{B_\alpha\}_{\alpha\in\hat A}$ be the essential covering associated with $\mathcal B$.

For a subset $X$ of $M$, let $X|_E$ denote $X\cap E$.
For $\alpha=\alpha_1\cdots\alpha_l\in\hat A$ and $1\le k\le l$, we prove by reverse induction on $k$ that
\[\beta(B_{\alpha_1\cdots\alpha_k}|_E)\le C(n,D).\]
The case $k=l$ is clear from Theorem \ref{thm:stab}(1).
Consider the case $k\le l-1$.
Recall that $\{B_{\alpha_1\cdots\alpha_{k+1}}\}_{\alpha_{k+1}=1}^{N_{k+1}(\alpha_1\cdots\alpha_k)}$ is a covering of an isotopic subball $\hat B_{\alpha_1\cdots\alpha_k}$ of $B_{\alpha_1\cdots\alpha_k}$.
Fix $(\alpha_1,\dots,\alpha_k)$ and set
\[B:=B_{\alpha_1\cdots\alpha_k},\quad\hat B:=\hat B_{\alpha_1\cdots\alpha_k},\quad B_\alpha:=B_{\alpha_1\cdots\alpha_k\alpha},\quad B_{\alpha}^i:=\lambda_iB_\alpha\]
for $1\le\alpha\le N_{k+1}(\alpha_1\cdots\alpha_k)$ and $0\le i\le m+1$.
Set $A^i:=\bigcup_{\alpha=1}^{N_{k+1}}B_{\alpha}^i$.
From the inclusions $\hat B|_E\subset A^0|_E\subset A^{m+1}|_E\subset B|_E$ and Theorem $\ref{thm:stab}$(2), we have
\[\beta(\hat B|_E)=\beta(B|_E)\le\rank\left[H_*(A^0|_E)\to H_*(A^{m+1}|_E)\right].\]
We estimate the right-hand side of the above inequality.
Let $\mu=(\gamma_1,\dots,\gamma_t)$ be such that $B_{\gamma_1}^i\cap\dots\cap B_{\gamma_t}^i\neq\emptyset$.
Suppose $B_{\gamma_s}$ has minimal radius among $\{B_{\gamma_j}\}_{j=1}^t$.
Then the following inclusions hold:
\[B_{\gamma_1}^i\cap\dots\cap B_{\gamma_t}^i\subset B_{\gamma_s}^i\subset\frac12B_{\gamma_s}^{i+1}\subset B_{\gamma_1}^{i+1}\cap\dots\cap B_{\gamma_t}^{i+1}.\]
Let $f_\mu^i:H_*((B_{\gamma_1}^i\cap\dots\cap B_{\gamma_t}^i)|_E)\to H_*((B_{\gamma_1}^{i+1}\cap\dots\cap B_{\gamma_t}^{i+1})|_E)$ be the inclusion homomorphism.
Then we have
\[\rank f_\mu^i\le\rank\left[H_*(B_{\gamma_s}^i|_E)\to H_*(\frac12B_{\gamma_s}^{i+1}|_E)\right]=\beta(B_{\gamma_s}|_E)\le C(n,D),\]
where Theorem \ref{thm:stab}(2) and the induction hypothesis are used.
By Lemma \ref{lem:rank}, we obtain
\[\rank\left[H_*(A^0|_E)\to H_*(A^{m+1}|_E)\right]\le (m+1)2^{C(n,D)}C(n,D).\]
This completes the induction step of the reverse induction.

Finally, since $E=\bigcup_{\alpha_1=1}^{N_1}B_{\alpha_1}|_E=\bigcup_{\alpha_1=1}^{N_1}B_{\alpha_1}^{m+1}|_E$, applying Lemma \ref{lem:rank} again, we conclude $\beta(E)\le C(n,D)$.
\end{proof}

Corollary \ref{cor:nonneg}(2) immediately follows from Lemma \ref{lem:nonneg}.

\begin{proof}[Proof of Corollary \ref{cor:nonneg}(2)]
Let $M$ be an $n$-dimensional Alexandrov space of nonnegative curvature, $E\subset M$ a (noncompact) extremal subset, and $p\in M$.
Then we can show that
\[\rank\left[H_*(B(p,D)\cap E)\to H_*(E)\right]\le C(n)\]
for any $D>0$.
Indeed, take an isotopy covering system for $B(p,D)\cap E$ instead of $E$, and repeat the above argument.
Only the last part of the proof is different: we estimate the rank of the inclusion homomorphism $H_*(B(p,D)\cap E)\to H_*(E)$ from the inclusions $B(p,D)\cap E\subset\bigcup_{\alpha_1=1}^{N_1}B_{\alpha_1}|_E\subset\bigcup_{\alpha_1=1}^{N_1}B_{\alpha_1}^{m+1}|_E\subset E$.
By rescaling, the constant $C(n)$ can be chosen independently from $D$.

Furthermore, Lemma \ref{lem:nonneg} and the fibration theorem imply that the inclusion $B(p,R)\cap E\hookrightarrow E$ is a homotopy equivalence for sufficiently large $R>0$.
Thus we have $\beta(E)\le C(n)$.
\end{proof}

\section{Volumes of extremal subsets}\label{sec:vol}

In this section, we prove Theorem \ref{thm:main}(3) and Corollary \ref{cor:precpt}.

Our main tool here is the gradient-exponential map of Perelman--Petrunin \cite{PP2}.
The reader is advised to first review \S\ref{sec:semi}.

First, we study local surjectivity of the restriction of the gradient-exponential map to an extremal subset.
Note that $\gexp_p|_{T_pE}:T_pE\to E$ is not surjective in general.
The local surjectivity of the gradient flow in an Alexandrov space was stated in \cite[\S2.2]{Pet2}.
We need its generalization to extremal subsets.

\begin{lem}[cf.\ {\cite[\S2.2 property (3)]{Pet2}}]\label{lem:flow}
Let $f:M\to\mathbb R$ be a semiconcave function on an Alexandrov space $M$ and $\Phi_f^t:M\to M$ its gradient flow (we assume that $\Phi_f^t$ is defined for all $x\in M$ and $t\ge0$).
Let $E$ be an extremal subset of $M$.
Then for any $y\in E$, there exist $x\in E$ and $t>0$ such that $\Phi_f^t(x)=y$.
\end{lem}

\begin{proof}
We prove it by induction on $\dim E$.
The case $\dim E=0$ is clear since every gradient flow fixes extremal points.
Suppose that the claim holds for extremal subsets of dimension $\le m-1$ and let $\dim E=m$.
Let $E^{(k)}$ be the $k$-dimensional stratum of the canonical stratification of $E$ (see \S\ref{sec:ex}).
Then the closure $\overline{E^{(k)}}$ is an extremal subset of dimension $\le m-1$.
By the induction hypothesis, the claim holds for all $y\in E\setminus E^{(m)}$.
Suppose that the claim does not hold for some $y\in E^{(m)}$.
Then $\Phi_f^t$ maps $E$ into $E\setminus\{y\}$ for all $t>0$.
Since $\Phi_f^t$ is homotopic to $\Phi_f^0=\mathrm{id}_M$, the following commutative diagram holds:
\[\xymatrix{
H_*(E,E\setminus\{y\})\ar[rr]^{\text{\normalsize$\mathrm{id}$}}\ar[dr]_{\text{\normalsize$(\Phi_f^t)_*$}}&&H_*(E,E\setminus\{y\})\\
&H_*(E\setminus\{y\},E\setminus\{y\})\ar[ur]_{\text{\normalsize$\iota_*$}}\ar@{}[u]|{\text{\LARGE$\circlearrowleft$}}&
}.\]
On the other hand, since $\dim E=m$, the top stratum $E^{(m)}$ is an $m$-dimensional topological manifold that is open in $E$.
Therefore, for $y\in E^{(m)}$, the local homology group $H_m(E,E\setminus\{y\})$ is nontrivial.
This contradicts the diagram.
\end{proof}

Using the above lemma, we give a sufficient condition for the gradient-exponential map restricted to an extremal subset to be locally surjective.
Note that we use the gradient-exponential map for the lower curvature bound $-1$ (see \S\ref{sec:semi}).

Recall that for $p,q\in M$, $q'_p\in\Sigma_p$ denotes one of the directions of minimal geodesics from $p$ to $q$.
Similarly, for a closed subset $A\subset M$, $A'_p\subset\Sigma_p$ denotes the set of all directions of minimal geodesics from $p$ to $A$.

\begin{prop}\label{prop:gexp}
Let $M$ be an Alexandrov space with curvature $\ge-1$, $E$ an extremal subset of $M$, and $p\in M$.
\begin{enumerate}
\item Suppose $\dist_p$ has no critical points on $\bar B(p,r)\setminus\{p\}$.
Then there exists $R>0$ such that
\[\gexp_p(B(o,R)\cap T_pE)\supset B(p,r)\cap E.\]
Note that if $B(p,r)\cap E\neq\emptyset$, then $p\in E$ (see Lemma \ref{lem:num}).
\item Suppose $\dist_p$ has no critical points on $A(p;r_1,r_2)$.
Then there exists $R>0$ such that
\[\gexp_p(A(o;r_1,R)\cap K((\partial B(p,r_1)\cap E)'_p))\supset A(p;r_1,r_2)\cap E,\]
where $K((\partial B(p,r_1)\cap E)'_p)$ is a subcone of $T_p=K(\Sigma_p)$.

Furthermore, if we define a map $G_p^{(r_1,R)}:B(p,r_1)\to B(p,R)$ by
\[G_p^{(r_1,R)}(x):=\gexp_p\left(\frac R{r_1}|px|x'_p\right)\]
for the above $R$, then we have
\[G_p^{(r_1,R)}(B(p,r_1)\cap E)\supset B(p,r_2)\cap E.\]
Note that this definition does not depend on the choice of $x'_p$ (see \S\ref{sec:semi}).
\end{enumerate}
\end{prop}

\begin{proof}
First we show (1).
Let us denote by $|\cdot|$ the norm on the tangent cone.
By the lower semicontinuity of $|\nabla\dist_p|$, there exists a constant $c>0$ such that $|\nabla\dist_p|>c$ on $\bar B(p,r)\setminus\{p\}$ (note that $|\nabla\dist_p|$ is close to $1$ near $p$).
Consider the semiconcave function $f=\cosh\circ\dist_p-1$, which satisfies the assumption of Lemma \ref{lem:flow} (i.e., the gradient flow of $f$ is defined for all points and time).
Let $z\in B(p,r)\cap E$.
It follows from Lemma \ref{lem:flow} by contradiction that there is a sequence $y_i\in E$ converging to $p$ such that $\Phi_f^{t_i}(y_i)=z$ for some $t_i$.
By reparametrization, $\gexp_p(s_i(y_i)'_p)=z$ for some $s_i$.
We show that $s_i$ is uniformly bounded above by some $R>0$.
Then, by compactness, we get $\gexp_p(s_0\xi_0)=z$ for some $s_0\le R$ and $\xi_0\in\Sigma_pE$, as desired.
Set $\beta_i(s)=\gexp_p(s(y_i)'_p)$.
From the differential equation \eqref{eq:rad}, we have
\[|p\beta_i(s)|'=\frac{\tanh|p\beta_i(s)|}{\tanh s}\cdot|\nabla_{\beta_i(s)}\dist_p|^2.\]
Together with the assumption $|\nabla\dist_p|>c$, this implies
\[\frac{|p\beta_i(s)|'}{\tanh|p\beta_i(s)|}\ge\frac{c^2}{\tanh s}.\]
Integrating this inequality over the interval $[\sigma,s_i]$ for some fixed $\sigma>0$, we obtain
\begin{equation}\label{eq:si}
\log\sinh s_i\le c^{-2}\left(\log\sinh r-\log\sinh|p\beta_i(\sigma)|\right)+\log\sinh\sigma.
\end{equation}
Since $|p\gexp_p(\sigma\,\cdot\,)|$ is a positive continuous function on $\Sigma_p$, we conclude that $s_i$ is uniformly bounded above.

Next we show (2).
The first statement follows from Lemma \ref{lem:flow} in the same way as (1).
Let us show the second.
Let $z\in B(p,r_2)\cap E$.
If $z\in A(p;r_1,r_2)\cap E$, then by the first statement, we have
\[\Phi_f^{t_0}(y_0)=\gexp_p(s_0(y_0)'_p)=z\]
for some $y_0\in\partial B(p,r_1)\cap E$, $t_0\ge0$, and $r_1\le s_0\le R$.
Even if $z\in B(p,r_1)\cap E$, the same equation holds for $y_0=z$, $t_0=0$, and $s_0=|pz|$.
It follows from Lemma \ref{lem:flow} by contradiction that for any $T>0$, there exists $x_0\in E$ such that $\Phi_f^T(x_0)=y_0$.
Suppose $T$ is sufficiently large (to be determined later) and consider the gradient curve $\alpha(t)=\Phi_f^t(x_0)$.
Then $G_p(\alpha(t))$ still lies on the curve $\alpha$ (for notational simplicity we write $G_p$ instead of $G_p^{(r_1,R)}$).
We show that $G_p(\alpha(t))$ is before $z$ on $\alpha$ when $t=0$ and after $z$ when $t=T$.
Then the claim follows from the intermediate value theorem.
Note that $G_p(\alpha(0))=G_p(x_0)$ and $G_p(\alpha(T))=G_p(y_0)$.
By the reparametrization \eqref{eq:re}, we can express $G_p(x_0)=\Phi_f^{\tau_0}(x_0)$, where
\begin{align*}
\tau_0&=\int_{|px_0|}^{\frac R{r_1}|px_0|}\frac{ds}{\tanh s\cosh|p\gexp_p(s(x_0)'_p)|}\\
&\le \int_{|px_0|}^{\frac R{r_1}|px_0|}\frac{ds}{\tanh s}\\
&\le \log\frac{\sinh R}{\sinh r_1}.
\end{align*}
Thus, if $T$ is sufficiently large, $G_p(x_0)$ is before $z=\Phi_f^{T+t_0}(x_0)$.
On the other hand, since $R\ge s_0$, $G_p(y_0)=\gexp(R(y_0)'_p)$ is after $z=\gexp_p(s_0(y_0)'_p)$.
This completes the proof.
\end{proof}

Next we provide a uniform estimate on the radius $R$ of Proposition \ref{prop:gexp} in terms of other geometric quantities.
The following technique was used in \cite[3.3]{PP2} to prove the convergence of the parameters of radial curves.
This controls the speeds of radial curves.

\begin{lem}\label{lem:gexp}
Let $M$ be an Alexandrov space with curvature $\ge-1$ and $p\in M$.
\begin{enumerate}
\item Assume that there exists $c>0$ such that $|\nabla\dist_p|>c$ on $\bar B(p,r)\setminus\{p\}$.
Assume further that there exist $\rho>0$ and $\theta<\pi/2$ such that $(\partial B(p,\rho))'_p$ is $\theta$-dense in $\Sigma_p$.
Then there exists $R=R(r,c,\rho,\theta)>0$ (depending only on $r$, $c$, $\rho$, and $\theta$) such that
\[\gexp_p^{-1}(B(p,r))\subset B(o,R).\]
\item  Assume that there exists $c>0$ such that $|\nabla\dist_p|>c$ on $A(p;r_1,r_2)$.
Assume further that there exist $\rho>0$ and $\theta<\pi/2$ such that for any $y\in\partial B(p,r_1)$ there is $x\in\partial B(p,\rho)$ with $\tilde\angle xpy<\theta$.
Then there exists $R=R(r_2,c,\rho,\theta)>0$ (independent of $r_1$) such that
\[\gexp_p^{-1}(B(p,r_2))\cap K((\partial B(p,r_1))'_p)\subset B(o,R).\]
\end{enumerate}
\end{lem}

\begin{proof}
(1) follows from (2) by taking $r_2=r$ and $r_1\to0$.
Let us show (2).
Consider the radial curve $\beta(s)=\gexp_p(sy'_p)$ for $y\in\partial B(p,r_1)$.
We must show that if $\beta(s)\in B(p,r_2)$ then $s<R(r_2,c,\rho,\theta)$.
Fix sufficiently small $\sigma>0$ depending only on $\rho$ and $\theta$ (to be determined later).
We first consider the case $r_1<\sigma$.
Since $|\dist_p|>c$ on $A(p;r_1,r_2)$, the same inequality as \eqref{eq:si} holds:
\begin{equation}\label{eq:s}
\log\sinh s\le c^{-2}\left(\log\sinh r_2-\log\sinh|p\beta(\sigma)|\right)+\log\sinh\sigma.
\end{equation}
Therefore, it is enough to show that $|p\beta(\sigma)|$ has a uniform positive lower bound depending only on $\rho$, $\theta$, and $\sigma$. 
By assumption, there exists $x\in\partial B(p,\rho)$ such that $\tilde\angle xpy<\theta$.
Proposition \ref{prop:rad}(1) implies $\tilde\angle xp\!\smile\!\beta(\sigma)<\theta$.
Therefore we have
\begin{align*}
|p\beta(\sigma)|&\ge\rho-|x\beta(\sigma)|\\
&=\cos\tilde\angle xp\!\smile\!\beta(\sigma)\cdot\sigma+o_\rho(\sigma)\\
&\ge\cos\theta\cdot\sigma+o_\rho(\sigma)\\
&\ge\const(\rho,\theta,\sigma)>0.
\end{align*}
Here $o_\rho(\sigma)$ is a function depending only on $\rho$ such that $o_\rho(\sigma)/\sigma\to0$ as $\sigma\to0$.
Since $\theta<\pi/2$, we can choose $\sigma=\sigma(\rho,\theta)$ so that the last inequality holds.

In the case $r_1\ge\sigma$, the same inequality as \eqref{eq:s} with $\sigma$ replaced by $r_1$ holds.
Since $|p\beta(r_1)|=r_1\ge\sigma$, the right-hand side is uniformly bounded above in terms of $r_2$, $c$, $\rho$, and $\theta$.
This completes the proof.
\end{proof}

In view of Lemma \ref{lem:gexp}, we modify Theorem \ref{thm:res} as follows.

\begin{thm}\label{thm:res2}
Suppose $(M_i,p_i)\in\mathcal A_p(n)$ converges to an Alexandrov space $(X,p)$ with dimension $\ge1$.
Then for sufficiently small $r>0$ and $c>0$, there exists $\hat p_i\in M_i$ converging to $p$ such that either (1) or (2) holds:
\begin{enumerate}
\item There is a subsequence $\{j\}\subset\{i\}$ such that $|\nabla\dist_{\hat p_j}|>c$ on $\bar B(\hat p_j,r)\setminus\{\hat p_j\}$
and $(\partial B(\hat p_j,r))'_{\hat p_j}$ is $(\pi/2-c)$-dense in $\Sigma_{\hat p_j}$.
\item There exists a sequence $\delta_i\to0$ such that
\begin{itemize}
\item[(i)] for any $\lambda>1$ and sufficiently large $i$, $|\nabla\dist_{\hat p_i}|>c$ on $A(\hat p_i;\lambda\delta_i,r)$ and for any $y\in\partial B(\hat p_i,\lambda\delta_i)$, there is $x\in\partial B(\hat p_i,r)$ with $\tilde\angle x\hat p_iy<\pi/2-c$;
\item[(ii)] for any limit $(Y,y_0)$ of  a subsequence of $(\frac1{\delta_i}M_i,\hat p_i)$, we have
\[\dim Y\ge\dim X+1.\]
\end{itemize}
\end{enumerate}
In particular, if $\dim X=n$, then (1) holds for all sufficiently large $i$.
\end{thm}

Note that (1) (resp.\ (2)(i)) states that the assumption of Lemma \ref{lem:gexp}(1) (resp.\ (2)) is satisfied with constants independent of $j$ (resp.\ $i$).

\begin{proof}
The proof is along the same lines as \cite[3.2]{Y}.
Fix $0<\varepsilon\ll\theta\le\pi/100$ (to be determined later).
Choose $0<r<1/100$ small enough that
\begin{itemize}
\item $\angle xpy-\tilde\angle xpy<\varepsilon$ for every $x,y\in\partial B(p,2r)$;
\item $(\partial B(p,2r))'_p$ is $\varepsilon$-dense in $\Sigma_p$.
\end{itemize}
Note that the latter implies that $|\nabla\dist_p|>1/10$ on $\bar B(p,r)\setminus\{p\}$.
Let $\{x_\alpha\}_\alpha$ be a maximal $\theta r$-discrete set in $\partial B(p,2r)$.
Furthermore, for each $\alpha$, let $\{x_{\alpha\beta}\}_{\beta=1}^{N_\alpha}$ be a maximal $\varepsilon r$-discrete set in $B(x_\alpha,\theta r)\cap\partial B(p,2r)$.
Then the Bishop--Gromov inequality implies that
\begin{equation}\label{eq:dimx}
N_\alpha\ge\const(X)\cdot\left(\frac\theta\varepsilon\right)^{\dim X-1}.
\end{equation}

Define functions $f_\alpha$ and $f$ on $X$ by
\[f_\alpha(x):=\frac1{N_\alpha}\sum_{\beta=1}^{N_\alpha}d(x_{\alpha\beta},x),\quad f(x):=\min_\alpha f_\alpha(x).\]
It is easy to see that $f$ has a strict maximum at $p$ on $\bar B(p,r)$, provided $\theta$ is small enough (see \cite[3.3]{Y}).

Fix a $\mu_i$-Hausdorff approximation $\varphi_i:B(p,1/\mu_i)\to B(p_i,1/\mu_i)$ with $\varphi_i(p)=p_i$, where $\mu_i\to0$ as $i\to\infty$.
Set $x_{\alpha\beta}^i:=\varphi_i(x_{\alpha\beta})$ and define functions $f_\alpha^i$ and $f^i$ on $M_i$ by
\[f_\alpha^i(x):=\frac1{N_\alpha}\sum_{\beta=1}^{N_\alpha}d(x_{\alpha\beta}^i,x),\quad f^i(x):=\min_\alpha f_\alpha^i(x).\]
Note that $f_\alpha^i$ and $f^i$ converge to $f_\alpha$ and $f$, respectively.
Let $\hat p_i$ be a maximum point of $f^i$ on $\bar B(p_i,r)$.
Then $\hat p_i$ converges to $p$, the unique maximum point of $f$.
Set
\[c:=\sin(\varepsilon/2N),\quad\text{where}\quad N=\max_\alpha N_\alpha.\]
Suppose that (1) does not hold for these $r$ and $c$.
Then for any sufficiently large $i$, there exists $y\in\bar B(\hat p_i,r)\setminus\{\hat p_i\}$ such that
\begin{itemize}
\item[(a)] $|\nabla_y\dist_{\hat p_i}|\le c$ or;
\item[(b)] $\tilde\angle x\hat p_iy\ge\pi/2-c$ for all $x\in\partial B(\hat p_i,r)$.
\end{itemize}
Let $\hat q_i\in\bar B(\hat p_i,r)\setminus\{\hat p_i\}$ be a farthest point from $\hat p_i$ satisfying either (a) or (b), and let $\delta_i$ be the distance between $\hat p_i$ and $\hat q_i$.
Then (2)(i) is obvious.
Moreover, $\delta_i\to 0$ since $|\nabla\dist_p|>1/10$ on $\bar B(p,r)\setminus\{p\}$ and $(\partial B(p,2r))'_p$ is $\varepsilon$-dense in $\Sigma_p$.

Let us show (2)(ii).
Suppose that $(\frac1{\delta_i}M_i,\hat p_i)$ converges to an Alexandrov space $(Y,y_0)$ of nonnegative curvature.
Passing to a subsequence, we may assume that $\hat q_i$ converges to $z_0\in Y$.
We may further assume that minimal geodesics $\hat p_i\hat q_i$ and $\hat p_ix_{\alpha\beta}^i$ converge to a minimal geodesic $y_0z_0$ and a ray $\gamma_{\alpha\beta}$ from $y_0$, respectively.
Let $v_i,v_{\alpha\beta}^i\in\Sigma_{\hat p_i}$ denote the directions of $\hat p_i\hat q_i$ and $\hat p_ix_{\alpha\beta}^i$, and let $v,v_{\alpha\beta}\in\Sigma_{y_0}$ be the directions of $y_0z_0$ and $\gamma_{\alpha\beta}$, respectively.
Note that
\[\angle(v_{\alpha\beta},v_{\alpha\beta'})\ge\tilde\angle x_{\alpha\beta}px_{\alpha\beta'}\ge\varepsilon/4\]
for every $1\le\beta\neq\beta'\le N_\alpha$.

First we show that
\begin{equation}\label{eq:v1}
\angle(v,v_{\alpha\beta})\ge\frac\pi2-\frac\varepsilon{2N}
\end{equation}
for every $\alpha$ and $\beta$.
If (a) holds for infinitely many $\hat q_i$, then by the lower semicontinuity of $|\nabla|$, we have $|\nabla_{z_0}\dist_{y_0}|\le\sin(\varepsilon/2N)$.
This implies that $\tilde\angle y_0z_0x_{\alpha\beta}(\infty)\le\pi/2+\varepsilon/2N$, where $x_{\alpha\beta}(\infty)$ denotes the element of the ideal boundary of $Y$ defined by the ray $\gamma_{\alpha\beta}$.
Thus we obtain
\[\angle(v,v_{\alpha\beta})\ge\tilde\angle z_0y_0x_{\alpha\beta}(\infty)\ge\frac\pi2-\frac\varepsilon{2N}.\]
On the other hand, if (b) holds for infinitely many $\hat q_i$, then by the monotonicity of angles, we have
\[\angle(v,v_{\alpha\beta})\ge\limsup_{i\to\infty}\tilde\angle\hat q_i\hat p_ix_{\alpha\beta}^i(r)\ge\frac\pi2-c\ge\frac\pi2-\frac\varepsilon{2N},\]
where $x_{\alpha\beta}^i(r)$ denotes the point on the minimal geodesic $\hat p_ix_{\alpha\beta}^i$ at distance $r$ from $\hat p_i$.
Therefore in either case we obtain \eqref{eq:v1}.

Next we fix $\alpha$ such that $(f^i)'_{\hat p_i}(v_i)=(f_\alpha^i)'_{\hat p_i}(v_i)$ for infinitely many $i$.
Since $f^i$ has a local maximum at $\hat p_i$, the first variation formula implies that
\[0\ge(f^i)'_{\hat p_i}(v_i)=\frac1{N_\alpha}\sum_{\beta=1}^{N_\alpha}-\cos\angle(v_i,v_{\alpha\beta}^i)\]
(choose $v_{\alpha\beta}^i$ so that the first variation formula holds for $v_i$).
Passing to the limit and using the lower semicontinuity of angles, we have
\begin{equation}\label{eq:v2}
0\ge\frac1{N_\alpha}\sum_{\beta=1}^{N_\alpha}-\cos\angle(v,v_{\alpha\beta}).
\end{equation}

Now, combining \eqref{eq:v1} and \eqref{eq:v2}, we obtain
\[\left|\angle(v,v_{\alpha\beta})-\frac\pi2\right|\le\varepsilon.\]
Hence $\{v_{\alpha\beta}\}_{\beta=1}^{N_\alpha}$ is an $\varepsilon/4$-discrete set of $A(v;\pi/2-\varepsilon,\pi/2+\varepsilon)$.
Since there exists a noncontracting map from $\Sigma_{y_0}$ to the unit sphere $S^{\dim Y-1}$ preserving the distance from $v$, we have
\begin{equation}\label{eq:dimy}
N_{\alpha}\le\const(n)\cdot\varepsilon^{-(\dim Y-2)}.
\end{equation}

Finally, combining \eqref{eq:dimx} and \eqref{eq:dimy} and taking $\varepsilon$ sufficiently small, we conclude that $\dim Y\ge\dim X+1$.
\end{proof}

We are now ready to prove the uniform boundedness of the volumes of extremal subsets.
Unlike the previous two sections, we prove it by contradiction using Theorem \ref{thm:res2}, without taking an essential covering.
However, the following proof has the same structure as the proof of the existence of an essential covering, Theorem \ref{thm:ess}, which was by contradiction using Theorem \ref{thm:res}.

Let $\vol_m$ denote the $m$-dimensional Hausdorff measure.
As noted in \S\ref{sec:main}, the Hausdorff measure of an extremal subset does not depend on whether the metric is intrinsic or extrinsic (\cite[3.17]{F}).

\begin{thm}\label{thm:vol}
For given $n$ and $D$, there exists a constant $C(n,D)$ satisfying the following:
Let $M\in\mathcal A(n)$, $p\in M$, and $E\subset M$ an $m$-dimensional extremal subset.
Then we have
\[\vol_m(B(p,D)\cap E)\le C(n,D).\]
\end{thm}

\begin{proof}
We use induction on $n$.
Suppose that the conclusion does not hold.
Choose a sequence of Alexandrov spaces $(M_i,p_i)\in\mathcal A_p(n)$ and $m$-dimensional extremal subsets $E_i\subset M_i$ such that $\vol_m(B(p_i,D)\cap E_i)\to\infty$ as $i\to\infty$.
We may assume that $(M_i,p_i)$ converges to an Alexandrov space $(X,p)$.
Set $k=\dim X$.
We prove by reverse induction on $k$ that there exists a constant $C$ such that
\[\vol_m(B(p_i,D)\cap E_i)\le C\]
for some subsequence.
This is a contradiction.
In what follows, $C$ denotes various positive constants independent of (sub)sequences.

First suppose $k=n$.
Take a finite covering $\{B(x_\alpha,r_\alpha/2)\}_{\alpha=1}^N$ of $\bar B(p,D)$, where $r_\alpha$ is the one of Theorem \ref{thm:res2}.
Then there exists $\hat x_\alpha^i\to x_\alpha$ for each $\alpha$ such that Theorem \ref{thm:res2}(1) holds for sufficiently large $i$.
Therefore, by Proposition \ref{prop:gexp}(1) and Lemma \ref{lem:gexp}(1), there exists $R_\alpha$ independent of $i$ such that
\[\gexp_{\hat x_\alpha^i}(B(o,R_\alpha)\cap T_{\hat x_\alpha^i}E_i)\supset B(\hat x_\alpha^i,r_\alpha)\cap E_i.\]
Since $\gexp_{\hat x_\alpha^i}$ is a $1$-Lipschitz map from the elliptic cone $(T_{\hat x_\alpha^i},\mathfrak h)$ over $\Sigma_{\hat x_\alpha^i}$ (see \S\ref{sec:semi}), we have
\[\vol_m(B(\hat x_\alpha^i,r_\alpha)\cap E_i)\le\int_0^{R_\alpha}\sinh^{m-1}r\cdot\vol_{m-1}(\Sigma_{\hat x_\alpha^i}E_i)\,dr\le C,\]
where the second inequality follows from the hypothesis of the induction on $n$.
Since $\{B(\hat x_\alpha^i,r_\alpha)\}_{\alpha=1}^N$ is a covering of $B(p_i,D)$ for sufficiently large $i$, we obtain
\[\vol_m(B(p_i,D)\cap E_i)\le\sum_{\alpha=1}^N\vol_m(B(\hat x_\alpha^i,r_\alpha)\cap E_i)\le C.\]

Next suppose $1\le k\le n-1$.
Cover $\bar B(p,D)$ by $\{B(x_\alpha,r_\alpha/2)\}_{\alpha=1}^N$ as above.
Then there exists $\hat x_\alpha^i\to x_\alpha$ for each $\alpha$ such that either (1) or (2) of Theorem \ref{thm:res2} holds.
If (1) holds, then we have $\vol_m(B(\hat x_\alpha^i,r_\alpha)\cap E_i)\le C$ for some subsequence as above.
Suppose that (2) holds for some $\alpha$.
We fix this $\alpha$ and omit it below.
Then there exists $\delta_i\to 0$ such that both (i) and (ii) holds.
Passing to a subsequence, we may assume that $(\frac1{\delta_i}M_i,\hat x^i)\xrightarrow{\mathrm{GH}}(Y,y_0)$.
Then we have $\dim Y\ge\dim X+1$.
Applying the hypothesis of the reverse induction to $\frac1{\delta_i}B(\hat x^i,2\delta_i)$ and $\frac1{\delta_i}E_i$, we have
\[\vol_m(B(\hat x^i,2\delta_i)\cap E_i)\le C\delta_i^m\]
for some subsequence.
Furthermore, by Proposition \ref{prop:gexp}(2) and Lemma \ref{lem:gexp}(2), there exists $R$ independent of $i$ such that
\[G_{\hat x^i}^{(2\delta_i,R)}(B(\hat x^i,2\delta_i)\cap E_i)\supset B(\hat x^i,r)\cap E_i.\]
Proposition \ref{prop:rad}(2) states that $G_{\hat x^i}^{(2\delta_i,R)}$ is $\frac{\sinh R}{\sinh2\delta_i}$-Lipschitz.
Together with the above inequality, this implies
\[\vol_m(B(\hat x^i,r)\cap E_i)\le\left(\frac{\sinh R}{\sinh2\delta_i}\right)^m\cdot C\delta_i^m\le C.\]
Since $\{B(\hat x_\alpha^i,r_\alpha)\}_{\alpha=1}^N$ is a covering of $B(p_i,D)$ for sufficiently large $i$, we obtain $\vol_m(B(p_i,D)\cap E_i)\le C$.

Finally, the case $k=0$ follows from the case $k\ge1$ by rescaling $M_i$ so that the new diameter is $1$.
This completes the proof.
\end{proof}

For a metric space $(X,d)$ and $\varepsilon>0$, we denote by $N_\varepsilon(X,d)$ the maximal number of $\varepsilon$-discrete points in $X$.
Here we allow the distance between two points to be infinite (since we consider not necessarily connected extremal subsets below).

\begin{thm}\label{thm:dis}
For given $n$ and $D$, there exists a constant $C(n,D)$ satisfying the following:
Let $M\in\mathcal A(n)$, $p\in M$, and $E\subset M$ an $m$-dimensional extremal subset.
Then for any $\varepsilon>0$, we have 
\[N_\varepsilon(B(p,D)\cap E,d_E)\le\frac{C(n,D)}{\varepsilon^m},\]
where $d_E$ denotes the induced intrinsic metric of $E$.
\end{thm}

The proof is similar to that of Theorem \ref{thm:vol}.
We can repeat the same argument by considering $\varepsilon^mN_\varepsilon(\,\cdot\,,d_E)$ instead of $\vol_m(\,\cdot\,)$.
Corollary \ref{cor:precpt} now follows from Theorem \ref{thm:dis} and Gromov's precompactness theorem \cite[5.2]{G2}.

\end{document}